\documentclass[a4paper,10pt]{amsart}
\usepackage{amssymb} \usepackage{amsfonts} \usepackage{amsmath}
\usepackage{amsthm} \usepackage{epsfig} 
\usepackage{amscd} \usepackage[all]{xy}
\usepackage{color}
\usepackage{enumerate}
\usepackage{url}
\usepackage{changepage}
\usepackage{booktabs}
\usepackage{array}

\usepackage{stmaryrd}
\usepackage{mathabx}
\usepackage{mathtools}
\usepackage{mathrsfs}

\usepackage{graphicx}
\usepackage{caption}
\usepackage{subcaption}

\usepackage[utf8x]{inputenc}

\makeatletter
\newcommand*{\@old@slash}{}\let\@old@slash\slash
\def\slash{\relax\ifmmode\delimiter"502F30E\mathopen{}\else\@old@slash\fi}

\def\bign#1{\mathclose{\hbox{$\left#1\vbox to8.5\p@{}\right.\n@space$}}\mathopen{}}
\def\Bign#1{\mathclose{\hbox{$\left#1\vbox to11.5\p@{}\right.\n@space$}}\mathopen{}}
\def\biggn#1{\mathclose{\hbox{$\left#1\vbox to14.5\p@{}\right.\n@space$}}\mathopen{}}
\def\Biggn#1{\mathclose{\hbox{$\left#1\vbox to17.5\p@{}\right.\n@space$}}\mathopen{}}

\makeatother

\def\backslash{\delimiter"526E30F\mathopen{}}

\newcommand{\C} {\ensuremath{\mathbb{C}}}

\newcommand{\R} {\ensuremath{\mathbb{R}}}
\newcommand{\Q} {\ensuremath{\mathbb{Q}}}
\newcommand{\Z} {\ensuremath{\mathbb{Z}}}
\newcommand{\I} {\ensuremath{\mathbb{I}}}

\newcommand{\M} {\ensuremath{\mathbb{M}}}

\newcommand{\B} {\ensuremath{\mathbb{B}}}

\renewcommand{\O} {\ensuremath{\textnormal{O}}}
\newcommand{\U} {\ensuremath{\textnormal{U}}}

\renewcommand{\H} {\ensuremath{\mathbf{H}}}

\newcommand{\g} {\ensuremath{\mathfrak{g}}}
\renewcommand{\k} {\ensuremath{\mathfrak{k}}}
\newcommand{\p} {\ensuremath{\mathfrak{p}}}

\renewcommand{\v} {\ensuremath{\mathfrak{v}}}

\newcommand{\su} {\ensuremath{\mathfrak{su}}}

\renewcommand{\sl} {\ensuremath{\mathfrak{sl}}}

\newcommand{\calO} {\ensuremath{\mathcal{O}}}

\newcommand{\calE} {\ensuremath{\mathcal{E}}}
\newcommand{\calL} {\ensuremath{\mathcal{L}}}

\newcommand{\calT} {\ensuremath{\mathcal{T}}}

\newcommand{\calA} {\ensuremath{\mathcal{A}}}

\newcommand{\de} {\ensuremath{\textnormal{d}}}

\newcommand{\Ad} {\ensuremath{\textnormal{Ad}}}
\newcommand{\ad} {\ensuremath{\textnormal{ad}}}

\newcommand{\SL} {\ensuremath{\textnormal{SL}}}
\newcommand{\Sp} {\ensuremath{\textnormal{Sp}}}
\newcommand{\SO} {\ensuremath{\textnormal{SO}}}
\newcommand{\SU} {\ensuremath{\textnormal{SU}}}
\renewcommand{\Sp} {\ensuremath{\textnormal{Sp}}}

\newcommand{\GL} {\ensuremath{\textnormal{GL}}}
\newcommand{\Ric} {\ensuremath{\textnormal{Ric}}}
\newcommand{\trace} {\ensuremath{\textnormal{trace}}}
\newcommand{\rk} {\ensuremath{\textnormal{rk}}}

\newcommand{\Hom} {\ensuremath{\textnormal{Hom}}}

\newcommand{\Isom} {\ensuremath{\textnormal{Isom}}}

\newcommand{\diag} {\ensuremath{\textnormal{diag}}}

\renewcommand{\Re} {\ensuremath{\mathcal{R}e}}
\renewcommand{\Im} {\ensuremath{\mathcal{I}m}}

\newcommand{\Vol} {\ensuremath{\textnormal{Vol}}}

\newcommand{\tX} {\ensuremath{\tilde{X}}}

\newcommand{\tx} {\ensuremath{\tilde{x}}}

\newcommand{\dev}[2] {\ensuremath{\frac{\partial{#1}}{\partial{#2}}}}

\newcommand{\scal}[2] {\ensuremath{\langle #1, #2 \rangle}}

\newcommand{\bigscal}[2] {\ensuremath{\big\langle #1, #2 \big\rangle}}

\newcommand{\Bigscal}[2] {\ensuremath{\Big\langle #1, #2 \Big\rangle}}

\newcommand{\GIT} {\ensuremath{/\!/}}

\newcommand*{\textrightarrow}[1]{\xrightarrow{\mathmakebox[1.5em]{#1}}}

\theoremstyle{plain}
\newtheorem{theorem}{Theorem}[section]
\newtheorem{lemma}[theorem]{Lemma}
\newtheorem{prop}[theorem]{Proposition}
\newtheorem{cor}[theorem]{Corollary}
\newtheorem{conj}[theorem]{Conjecture}

\newtheorem*{theorem*}{Theorem}
\newtheorem*{lemma*}{Lemma}
\newtheorem*{prop*}{Proposition}
\newtheorem*{cor*}{Corollary}

\theoremstyle{remark}

\newtheorem{remark}[theorem]{Remark}
\newtheorem*{remark*}{Remark}

\theoremstyle{definition}

\newtheorem{defn}[theorem]{Definition}

\newtheorem*{defn*}{Definition}

\usepackage[english]{babel}

\title[Rigidity of maximal holomorphic representations]{Rigidity of maximal holomorphic representations of Kähler groups.}
\author{Marco Spinaci}

\begin{document}

\begin{abstract}
 We investigate representations of Kähler groups $\Gamma = \pi_1(X)$ to a semisimple non-compact Hermitian Lie group $G$ that are deformable to a representation admitting an (anti)-holomorphic equivariant map. Such representations obey a Milnor--Wood inequality similar to those found by Burger--Iozzi and Koziarz--Maubon. Thanks to the study of the case of equality in Royden's version of the Ahlfors--Schwarz Lemma, we can completely describe the case of maximal holomorphic representations. If $\dim_{\C}X \geq 2$, these appear if and only if $X$ is a ball quotient, and essentially reduce to the diagonal embedding $\Gamma < \SU(n,1) \to \SU(nq,q) \hookrightarrow \SU(p,q)$. If $X$ is a Riemann surface, most representations are deformable to a holomorphic one. In that case, we give a complete classification of the maximal holomorphic representations, that thus appear as preferred elements of the respective maximal connected components.
\end{abstract}

\maketitle

\section{Introduction}

 Consider the character variety of representations of a finitely presented group $\Gamma$ to a Lie group $G$:
 $$
 \M = \Hom(\Gamma, G)^{ss}/G,
 $$
 where $ss$ stands for semisimple, that is, representations $\rho \colon \Gamma \to G$ such that the Zariski closure of the image $\overline{\rho(\Gamma)}$ is reductive. When studying the topology of $\M$, the most basic step is being able to distinguish different connected components; to that aim, one is naturally led to look for characteristic numbers, invariant under continuous deformations of a representation. When $\Gamma = \pi_1(X)$ is the fundamental group of a Kähler manifold $(X, \omega_X)$ and $G$ is of Hermitian type, i.e. $Y = G/K$ has a $G$-invariant Kähler form $\omega_Y$, a natural candidate is the \emph{Toledo invariant}, which is defined as 
 \begin{equation}\label{eqn:toledointro}
  \tau(\rho) = \frac{1}{n!} \int_X f^*\omega_Y \wedge \omega_X^{n-1}.
 \end{equation}
 This has been widely studied in the case $\dim_\C(X) = 1$, i.e., for surface groups. In the case $G = \SL(2,\R)$, this is equivalent to considering the Euler class, and it satisfies the Milnor--Wood inequality, see \cite{Mi58} and \cite{Wo71}. The study of the case of equality in this inequality is due to Goldman, \cite{Go80}. The generalizations of these results for more general Hermitian Lie groups $G$ have formed a very active field of study in the last 30 years: The relevant Milnor--Wood type inequality is due to Domic and Toledo \cite{DoTo87} for most classical groups $G$, and to Clerc and \O rsted in general \cite{ClOr03}. The study of the maximal case has experienced a longer history: Toledo \cite{To89} has analyzed $G = \SU(p,1)$, Hernández \cite{He91} $G = \SU(p,2)$ and Bradlow, García-Prada and Gothen \cite{BrGPGo03} $G = U(p,q)$, where they computed the number of the maximal connected components. They completed the same program considering in a series of consecutive works the other classical groups, see \cite{BrGPGo06} for a survey. More recently, Burger, Iozzi and Wienhard \cite{BuIoWi10} completed the picture with very general results about geometric properties of maximal representations; in particular they proved that maximal representations are faithful, discrete, semisimple and fix a ``tube type'' subdomain.

 When $n = \dim_\C X \geq 2$, the results are much more partial. The Milnor--Wood type inequality has been proved by Burger and Iozzi \cite{BuIo07} in the case where $X$ is locally symmetric, i.e. the universal cover $\tX = G'/K'$ is itself a Hermitian symmetric space, using bounded cohomology techniques. In that case, the Milnor--Wood inequality reads:
 \begin{equation}\label{eqn:MWBI}
  \big|\tau(\rho)\big| \leq \frac{\rk(G)}{\rk(G')} \Vol(X).
 \end{equation}
 In this symmetric case, the study of maximal representations is especially interesting for $G' = \SU(n,1)$, since in higher rank Mostow's superrigidity applies. Because of that, a maximal representation of a higher rank lattice is induced by a ``tight homomorphism'' of $G'$ to $G$, and those have been completely classified (see for example \cite{Ha12}). In the case of rank one, maximal representations are expected to be extremely special, too. Indeed, there is the following conjecture:
 \begin{conj}
  Let $\Gamma < \SU(n,1)$ be a cocompact complex hyperbolic lattice, $n > 1$. Suppose that $\rho \colon \Gamma \to G$ is a maximal representation to a simple Lie group of Hermitian type $G$, i.e. $\tau(\rho) = \rk(G) \Vol(\Gamma \backslash \B^n)$. Then in fact $G = \SU(p,q)$ with $p \geq nq$ and $\rho$ is a ``trivial deformation'' of the standard diagonal embedding
  \begin{equation}\label{eqn:standard}
   \rho_{std} \colon \Gamma < \SU(n,1) \hookrightarrow \SU(nq,q) \hookrightarrow \SU(p,q),
  \end{equation}
  i.e. $\rho(\gamma) = \chi(\gamma) \rho_{std}(\gamma)$ where $\chi \colon \Gamma \to Z_G\big(\rho_{std}(\SU(n,1))\big)$.
 \end{conj}
 Remark that in the conjecture we restrict without loss of generality to $G$ simple and $\tau(\rho) \geq 0$. Indeed, the Toledo invariant is additive on the factors of a decomposition into irreducible factors of $Y$, and exchanging the complex structure on an irreducible $Y$ changes the sign of $\tau$, hence in fact this covers all the interesting cases. At the present day, the conjecture is only known for $\rk(G) \leq 2$ (and $G \neq \SO^*(10)$), thanks to Koziarz and Maubon \cite{KoMa08}, who used Higgs bundles techniques to reprove \eqref{eqn:MWBI} in this case and to study the equality case. Very recently Pozzetti \cite{Po14} proved the conjecture (without the cocompactness hypotheses) for representations $\rho$ such that the Zariski closure of $\rho(\Gamma)$ does not contain any factor of the form $\SU(k,k)$.
 
 For more general Kähler manifolds, not even a Milnor--Wood inequality is available. The most relevant work here is again due to Koziarz and Maubon \cite{KoMa10}, who considered a slight variation of the Toledo invariant for complex varieties of general type, and were able to replicate the results in \cite{KoMa08} in this broader setting. In this paper we prove that both the inequality and the study of the case of equality can be proved if one knows beforehand that the representation admits a $\rho$-equivariant \emph{holomorphic} (or anti-holomorphic) map:
\begin{theorem}\label{thm:main}
 Let $X$ be a compact Kähler manifold, $n = \dim_\C X$, $\Gamma = \pi_1(X)$ its fundamental group and $\rho \colon \Gamma \to G$ a representation to a noncompact Hermitian Lie group, such that $\rho$ can be deformed to one admitting an (anti)-holomorphic equivariant map. Let $k \leq 0$ be a real number such that the Ricci curvature of $X$ is bounded below by $k$. Then the following Milnor--Wood inequality holds:
 \begin{equation}\label{eqn:MWio}
 \big|\tau(\rho)\big| \leq \tau_{\max} = \frac{-2k}{n+1} \rk(G) \Vol(X).
 \end{equation}
 If, furthermore, $\rho$ is maximal, i.e. equality holds in \eqref{eqn:MWio}, $n > 1$ and $G$ is a simple linear group, then $X$ is a ball quotient, $G = \SU(p,q)$ with $p \geq nq$ and $\rho$ is a trivial deformation of $\rho_{std}$.
\end{theorem}
 The inequality \eqref{eqn:MWio} follows directly from a result by Royden (see \cite{Ro80}, Theorem 1). The study of the case of equality descends from the study of equality in its inequality, together with a careful study of the geometry of the classical Lie groups of Hermitian type. The techniques are not new: They have been used by Koziarz and Maubon in \cite{KoMa08}, in the case where $\tX = \B^n$ and $G = \SU(p,q)$, where an easier version of Royden's theorem applies (see loc. cit. Theorem 2). For manifolds endowed with a Kähler-Einstein metric, the inequality \eqref{eqn:MWio} is compatible with the one in \cite{KoMa10}, but not as strict as \eqref{eqn:MWBI} (indeed, the identity morphism for higher rank groups realizes the equality in \eqref{eqn:MWBI}, but that is clearly not true in \eqref{eqn:MWio}). Also remark that some of the results of this theorem are reminiscent of the ones obtained by Eyssidieux in \cite{Ey99}, pages 84 ff., studying the case of equality in Arakelov inequalities. Also remark that this theorem gives a different proof of a part of the statements in \cite{Ha11}: In particular, together with \cite{Ha12}, this gives
 \begin{cor}\label{cor:standard}
  Let $\Gamma < \SU(n,1)$, $n > 1$, be a complex hyperbolic cocompact lattice. Every maximal representation that can be deformed to one that factors through a representation of $\SU(n,1)$,
  $$
   \rho \colon \Gamma \hookrightarrow \SU(n,1) \to G
  $$
 is a trivial deformation of the standard one \eqref{eqn:standard}.
 \end{cor}

 The above techniques also give non-trivial results when $n=1$. In this case, thanks to the series of works by Bradlow, García-Prada and Gothen (see \cite{BrGPGo06} and the references therein), we know that being deformable to an (anti)-holomorphic representation is not very restrictive: For $G = \SU(p,q)$ or $\SO^*(2n)$ or $\SO_0(n,2)$, with $n \geq 4$, this is always the case; for $G = \Sp(2n,\R)$, $n \geq 3$, this is true unless $\rho$ is in the ``Hitchin component''. Since (anti)-holomorphic representations can be characterized as minimizers of the Morse function (see Proposition \ref{prop:energyinequality} for the precise statement), these are always the minima in each connected components. In the Hitchin case for $\Sp(2n, \R)$, such minima are represented by Fuchsian representations, see \cite{GPGoMR13}. The authors have used such results to compute the number of maximal connected components. Using the same techniques as in Theorem \ref{thm:main}, and making use of the above holomorphicity result for the minima in these connected components, we give an alternative way to compute such a number, together with an explicit classification of the representations realizing those minima.

\begin{theorem}\label{thm:classification}
 Let $\Gamma_g = \pi_1(\Sigma_g)$ be a surface group of genus $g \geq 2$. Let $\rho \colon \Gamma \to G$ be a maximal representation, where $G$ is either $\SU(p,q)$, $\SO^*(2n)$ or $\SO_0(n, 2)$, $n \geq 4$. Then $\rho$ can be deformed to one of the holomorphic representations $\rho_{tot}$ in Table \ref{tab:representations} or to one of its ``trivial deformations''. These differ from $\rho_{tot}$ by multiplication for a $\chi \colon \Gamma_g \to Z$, where $Z$ is the centralizer of $\rho_{tot}(\SL_2(\R))$ as in Table \ref{tab:centralizers}. If $G = \Sp(2n,\R)$, $n \geq 3$, either the above is true, or $\rho$ is in the Hitchin component, hence it can be deformed to a Fuchsian representation.
\end{theorem}
\begin{table}[ht]
 \caption{Canonical representatives of maximal holomorphic representations (convention: $a, b, c, d,$ are real, $\alpha, \beta$ are complex; the use $\SL_2(\R)$ or $\SU(1,1)$ is deduced from the notation).}
 \label{tab:representations}
\begin{adjustwidth}{-0.9in}{-0.9in}
\small
 \begin{tabular}{m{0.4in} m{2.7in} c}
  \centering$G$ & \centering $f_* \colon \sl_2(\R) \cong \su(1,1) \to \g$ & $\rho_{tot} \colon \SL_2(\R) \cong \SU(1,1) \to G$\\[0.2cm]
  \midrule
  \centering $\SU(p,q)$, $p \geq q$ & \centering $\begin{pmatrix}ia&\beta\\\bar\beta&-ia\end{pmatrix} \mapsto \begin{pmatrix} iaI_q&0&\beta I_q\\0&I_{p-q}&0\\
                                                                                  \bar\beta I_q&0&-iaI_q
                                                                                 \end{pmatrix}$
                                                                                 &
                                                                                 $\begin{pmatrix}
                                                                                   \alpha & \beta\\
                                                                                   \bar\beta&\bar\alpha
                                                                                  \end{pmatrix} \mapsto
                                                                                  \begin{pmatrix}
                                                                                   \alpha I_q & 0 & \beta I_q\\
                                                                                   0 & I_{p-q} & 0\\
                                                                                   \bar\beta I_q & 0 & \bar\alpha I_q
                                                                                  \end{pmatrix}$\\[0.7cm]
 \centering $\Sp(2n,\R)$, $n \geq 3$ & \centering $\begin{pmatrix}
                              a&b\\
                              c&-a
                             \end{pmatrix}
                             \mapsto
                             \begin{pmatrix}
                              a I_n & b I_n\\
                              c I_n & -a I_n
                             \end{pmatrix}$
                          & $\begin{pmatrix}
                              a&b\\
                              c&d
                             \end{pmatrix}
                             \mapsto
                             \begin{pmatrix}
                              a I_n & b I_n\\
                              c I_n & d I_n
                             \end{pmatrix}$\\[0.5cm]
 \centering $\SO_0(n,2)$, $n \geq 4$ & \centering $\begin{pmatrix}
                             ia & b-ic\\
                             b+ic & -ia
                            \end{pmatrix}
                            \mapsto
                            \begin{pmatrix}
                             0 & 0 & \dots & 2b & 2c\\
                             0 & 0 & \dots & 0 & 0\\
                             \vdots & \vdots & \ddots & \vdots & \vdots\\
                             2b & 0 & \dots & 0 & 2a\\
                             2c & 0 & \dots & -2a & 0
                            \end{pmatrix}$
                          & $\begin{array}{c}\begin{pmatrix}
                              \alpha & \beta\\
                              \bar\beta & \bar\alpha
                             \end{pmatrix}
                             \mapsto\\
			     \begin{pmatrix}
			      2|\beta|^2+1 & 0 & 2 \Re(\alpha\bar\beta) & 2 \Im(\alpha \bar\beta)\\
			      0 & I_{n-1} & 0 & 0\\
			      2 \Re(\alpha\beta) & 0 & \Re(\alpha^2 + \beta^2) & \Im(\alpha^2 + \beta^2)\\
			      -2 \Im(\alpha\beta) & 0 & -\Im(\alpha^2 + \beta^2) &\Re(\alpha^2-\beta^2)
			     \end{pmatrix}\end{array}$\\[1.2cm]
 \centering $\SO^*(2n)$ & \centering $\begin{pmatrix}
                 ia & b-ic\\
                 b+ic & -ia
                \end{pmatrix}
                \mapsto
		\begin{pmatrix}
		 ibJ & a I_n + ic J\\
		 -aI_n + icJ & -ibJ
		\end{pmatrix}
		$,
		$J = \begin{pmatrix}
		      0 & I_n\\
		      -I_n & 0
		     \end{pmatrix}$
 \end{tabular}
\normalsize
\end{adjustwidth}
\end{table}

\begin{table}[ht]
 \caption{Centralizers and number of connected components}
 \label{tab:centralizers}
\small
  \begin{tabular}{m{0.4in} m{3.2in} p{0.75in}}
   \centering $G$ & \centering $Z_G(\rho_{tot}(\SL_2(\R)))$ & Number of connected components\\
   \midrule
   $\SU(p,p)$ & \centering $\Bigg\{ \begin{pmatrix}U&0\\0&U\end{pmatrix} : U \in U(p), \det(U) = \pm1 \Bigg\} \cong U(p) \rtimes \Z/2\Z$
              & $2^{2g}$\\[0.6cm]
   $\SU(p,q)$, $p > q$ & \centering $\Bigg\{ \begin{pmatrix}U&0&0\\0&F&0\\0&0&U\end{pmatrix} : \begin{array}{l} U \in U(q), F \in U(p-q),\\ \det(U)^2\det F =1 \end{array}\Bigg\}$
              & 1\\[0.8cm]
   $\Sp(2n, \R)$, $n \geq 3$ & \centering $\Bigg\{ \begin{pmatrix} Q & 0\\0& Q\end{pmatrix} : Q \in O(n) \Bigg\}$
              & $\begin{array}{l} 2^{2g+1} \text{(plus }2^{2g}\\ \text{Hitchin ones}\\ \text{\cite{BrGPGo06})}\end{array}$\\[0.4cm]
   $\SO_0(n,2)$, $n \geq 4$ & \centering $\Bigg\{ \begin{pmatrix} \det P\\&P\\&&\det P\\&&&\det P \end{pmatrix} : P \in O(n-1) \Bigg \}$
              & $2^{2g+1}$\\[1cm]
   $\SO^*(2n)$ & \centering $\Sp(n) \subset U(n) \subset \SO^*(2n)$ & 1
  \end{tabular}
\normalsize
\end{table}

We remark here that the ``Cayley correspondence'' discovered by Bradlow, García-Prada and Gothen, relating representations (or Higgs bundles) in the symplectic or orthogonal group (see, for example, \cite{BrGPGo13}) is reflected in Table \ref{tab:centralizers} by the centralizers of these representations.

Finally, remark that Theorem \ref{thm:main} seems to suggest that $\tau(\rho)$ should always vanish if the Ricci curvature of $X$ is non-negative. This is indeed the case, as it follows from a recent theorem of Biswas and Florentino \cite{BiFl14} (together with the fact, due to Milnor \cite{Mi58}, that non-negative Ricci curvature implies virtually nilpotent fundamental group):
\begin{prop}\label{prop:nilpotent}
 Let $X$ be a compact Kähler manifold such that $\Ric(X) \geq 0$ or, more generally, such that $\pi_1(X)$ is virtually nilpotent. Then for every Hermitian Lie group $G$ and any $\rho \colon \Gamma \to G$, $\tau(\rho) = 0$.
\end{prop}
An immediate corollary of this fact (plus a Theorem of Delzant \cite{De10}) is that no cocompact lattice of a Hermitian Lie group is solvable. This fact is trivial in rank 1 (in that case, $\Gamma$ is hyperbolic), and it follows from the Margulis Normal Subgroup in higher rank.

\subsection{Organization of the paper}
We give different definitions of the Toledo invariant $\tau(\rho)$, using the Higgs bundles formalism, in Section \ref{sec:definitions}, together with its relation to the energy $E(\rho)$ and to the existence of an (anti)-holomorphic $\rho$-equivariant map. Section \ref{sec:royden} is devoted to the statement of Royden's version of the Ahlfors-Schwarz-Pick Lemma, and to the study of the case of equality. The geometric part of the proof of Theorem \ref{thm:main} occupies Section \ref{sec:proof}, while all the technical part involving matrix computations is postponed to Section \ref{sec:prooflemma}. The proof of Theorem \ref{thm:classification} occupies Section \ref{sec:proofclass}. Finally, in Section \ref{sec:other} we prove Proposition \ref{prop:nilpotent}, and discuss some future directions aiming to prove deformability to holomorphic representations.

\subsection{Acknowledgements}
I would like to thank Beatrice Pozzetti, Vincent Koziarz and Julien Maubon for some fruitful exchanges by email as well as Oskar Hamlet for explaining his results to me.

\section{Definitions and generalities}\label{sec:definitions}

Let $(X, \omega_X)$ be an $n$-dimensional compact Kähler manifold, $\Gamma = \pi_1(X, x_0)$ its fundamental group and $G$ a connected semisimple Lie group of non-compact Hermitian type. Denote by $K$ a maximal compact subgroup of $G$ and by $(Y, \omega_Y)$ the associated symmetric space, together with the standard Kähler form.

\begin{defn}
 The \emph{Toledo invariant} of $\rho$ is defined as
 \begin{equation}\label{eqn:toledo}
 \tau(\rho) = \frac{1}{n!} \int_X f^*\omega_Y \wedge \omega_X^{n-1},
 \end{equation}
 where $f \colon \tX \to Y$ is any $\rho$-equivariant continuous map from the universal cover $\tX$ of $X$ to $Y$.
\end{defn}

Such maps always exist (and can indeed be taken smooth), since $Y$ is contractible. For the same reason, any two of them are homotopic, and this implies that the Toledo invariant is independent of the chosen $f$.

\begin{lemma}
 $\tau$ defined as in \eqref{eqn:toledo} invariant under deformation of the representation $\rho$ and under conjugation. 
\end{lemma}
\begin{proof}
 The proof of this Lemma is classical. If $\tilde \rho = \Ad_g \rho$, for some $g \in G$, and if $f$ is $\rho$-equivariant, then $\tilde f = g \cdot f$ is $\tilde \rho$ equivariant, and $\tilde f^*\omega_Y = f^*\omega_Y$ by left invariance of $\omega_Y$. To prove invariance under deformation, remark that the only term depending on $\rho$ in the definition of $\tau(\rho)$ is $f^*\omega_Y$. Suppose now that $Y$ is irreducible. Then, there is an integer $c_Y$ such that:
 \begin{equation}\label{eqn:KE}
 \omega_Y = \frac{4 \pi}{c_Y} c_1(K_{\check{Y}}),
 \end{equation}
 where $K_{\check{Y}}$ is the restriction to $Y$ of the canonical bundle of the compact dual $\check{Y}$ of $Y$ (for the explicit values of $c_Y$, see the table in \cite{KoMa10}, which is taken from \cite{He78} and \cite{Li06}). In particular, $f^*c_1(K_{\check{Y}})$ is an \emph{integral} cohomology 2-class. Since this varies continuously with $\rho$, it must actually be constant on connected components of $\Hom(\Gamma, G)$. As a consequence, the cup product with the fixed $2n-2$ cohomology class $[\omega_X^{n-1}]$ must be constant, as well. The general case follows from additivity of $\tau$ with respect to the decomposition of $Y$ into irreducible factors.
\end{proof}

The immediate consequence of this lemma is that $\tau$ is constant on every connected component of the quotient space $\Hom(\Gamma, G)/G$, where $G$ acts on $\Hom(\Gamma, G)$ by conjugation. In order to work on a separated space, we will actually consider the GIT quotient
$$
\M = \Hom(\Gamma, G)\GIT G \cong \Hom(\Gamma, G)^{ss}/G,
$$
which goes under the name of ``$G$-character variety'' or ``Betti moduli space''. Its points are in bijection with the orbits of \emph{semisimple} representations $\rho$ (i.e. representations such that the Zariski closure of $\rho(\Gamma)$ is a reductive subgroup of $G$).

\begin{defn}\label{defn:energy}
 The \emph{energy} of a representation $\rho \colon \Gamma \to G$ is defined as
 \begin{equation}\label{eqn:energy}
 E(\rho) = \inf \bigg\{ \frac{1}{2} \int_X \big\| \de f \big\|^2 \frac{\omega_X^n}{n!} \ \Big | \ f \text{ is smooth and } \rho \text{-equivariant} \bigg\}.
 \end{equation}
 Here, $\de f$ is seen as a section of the bundle $T^*\tX \otimes f^*TY$, endowed with the metric induced by the Riemannian metrics on $X$ and $Y$. Since $G$ acts by isometries, the norm of $\de f$ actually descends to a function on $X$.
\end{defn}

Recall that a map $f$ realizing the minimum in \eqref{eqn:energy} is called \emph{harmonic}. By Corlette's theorem \cite{Co88, JoYa91}, such a map exists if and only if $\rho$ is semisimple. Given a harmonic $\rho$-equivariant map $f$ and a faithful linear representation of $G$, one can interpret $f$ as a metric on the flat complex bundle corresponding to $\rho$. In that way, to a given semi-simple representation $\rho$, one can associate a \emph{Higgs bundle} $(\calE, \Phi)$ as in \cite{Si92}. This correspondence is well defined and, under a suitable notion of stability and up to some isomorphisms, bijective. Under this mapping, $\Phi$ is essentially the projection of $\de f$ to the holomorphic tangent bundle $T^{1,0}\tX$. In particular, one finds:
\begin{equation}\label{eqn:energyhiggs}
E(\rho) = \big\|\Phi\big\|_{L^2}^2.
\end{equation}

\begin{lemma}\label{lemma:toledohiggs}
 Let $\rho \colon \Gamma \to G$ be a semisimple representation to a simple Lie group of non-compact Hermitian type $G$. Denote by $Z$ the generator of the center of the Lie algebra $\k$ of $K$ having eigenvalues $\pm i$, so that $J = \ad(Z)$ gives the complex structure on $Y$. Let $\scal{\cdot}{\cdot}$ be the $G$-invariant metric on the Lie algebra $\g$ of $G$ inducing the chosen metric on $Y$ and extend it to a Hermitian metric on $\g_\C$, still denoted by $\scal{\cdot}{\cdot}$. Then:
  \begin{equation}\label{eqn:toledohiggs}
    \tau(\rho) = \frac{1}{n}\int_X \bigscal{\Phi}{[\Phi, iZ]} \frac{\omega_X^n}{n!}.
  \end{equation}
\end{lemma}
\begin{proof}
 Since $\rho$ is semisimple, we can, and will, fix a harmonic $\rho$-equivariant map $f$. First observe that $*\omega_X = \frac{1}{(n-1)!} \omega_X^{n-1}$, hence
 $$
 \tau(\rho) = \frac{1}{n!} \int_X f^*\omega_Y \wedge \omega_X^{n-1} = \frac{1}{n} \int_X \scal{f^*\omega_Y}{\omega_X} \frac{\omega_X^n}{n!} = \frac{1}{n}\sum_j \int_X \omega_N\big(\dev{f}{x_j}, \dev{f}{y_j}\big) \frac{\omega_X^n}{n!},
 $$
 where, abusing notation, we are denoting by $\dev{}{x_j}$, $\dev{}{y_j}$ a local orthonormal frame on $X$ compatible with the complex structure $J$, i.e. $J(\dev{}{x_j}) = \dev{}{y_j}$. Denote by $\beta_Y$ the $\g$-valued 1-form on $Y$ giving the usual inclusion of vector bundles $TY \subset Y \times \g$, that is a right inverse to the projection $Y \times \g \ni (y, \xi) \mapsto \dev{}{t} \big(\exp(t\xi)\cdot y\big)\big|_{t=0}$. Then:
 \begin{align*}
 \tau(\rho) &= \frac{1}{n} \sum_j \int_X \Bigscal{\big[Z,\beta_Y(\dev{f}{x_j})\big]}{\beta_Y(\dev{f}{y_j})} \frac{\omega_X^n}{n!}\\
            &= \frac{1}{n} \sum_j \int_X \Bigscal{\big[Z,\Phi(\partial_j)+\Phi^*(\bar\partial_j)\big]}{i\big(\Phi(\partial_j)-\Phi^*(\bar\partial_j)\big)},
 \end{align*}
 where we have used that $f^*\beta_Y = \Phi + \Phi^*$ and denoted by $\partial_j = \frac{1}{2}\dev{}{x_j} -\frac{i}{2}\dev{}{y_j}$. Now $Z$ is skew-adjoint; extending $\scal{\cdot}{\cdot}$ to a Hermitian form, and since $\Phi^*(\bar\partial_j)$ is the adjoint of $\Phi(\partial)$, we have
\begin{align*}
\bigscal{[Z,\Phi(\partial)]}{i\Phi^*(\bar\partial)} &= \bigscal{\Phi(\partial)}{[-Z,i\Phi^*(\bar\partial)]} = \bigscal{[-i\Phi(\partial), Z]}{\Phi^*(\bar\partial)}\\
                                                    &= \bigscal{[\Phi(\partial),Z]}{i\Phi^*(\bar\partial)},
\end{align*}
hence this vanishes. Furthermore,
\begin{align*}
\bigscal{[Z,\Phi^*(\bar\partial)]}{-i\Phi^*(\bar\partial)} &= \bigscal{i\Phi(\partial)}{[\Phi(\partial), -Z]} = \bigscal{[-iZ, \Phi(\partial)]}{\Phi(\partial)}\\
                                                           &= \bigscal{[Z, \Phi(\partial)]}{i\Phi(\partial)}.
\end{align*}
We thus obtain the wished result:
$$
\tau(\rho) = \frac{2}{m} \sum_j \int_X \bigscal{[Z,\Phi(\partial_j)]}{i\Phi(\partial_j)} \frac{\omega_X^n}{n!} = \frac{1}{n} \int_X \bigscal{[Z,\Phi]}{i\Phi} \frac{\omega_X^n}{n!}.
$$
\end{proof}

Given any $\rho$, one can define its semi-simplification $\rho^{ss}$ as any element in the unique closed orbit contained in the closure of the orbit of $\rho$. Recall the following:
\begin{lemma}[\cite{Sp14}, Lemma 3.4]\label{lemma:energysemisimplification}
 Let $\rho \colon \Gamma \to G$ be a representation. Then $E(\rho) = E(\rho^{ss})$.
\end{lemma}

We give one last definition:
\begin{defn}
 A semisimple representation $\rho \colon \Gamma \to G$ will be called $\pm$-holomorphic if one (hence, any) harmonic $\rho$-equivariant map $f$ is (anti)-holomorphic.
\end{defn}

We are ready to state the result linking the different definitions given so far:
\begin{prop}\label{prop:energyinequality}
 For every representation $\rho \colon \Gamma \to G$, we have $E(\rho) \geq n|\tau(\rho)|$. Furthermore, $\rho$ is $\pm$-holomorphic if and only if it is semisimple and equality holds.
\end{prop}
\begin{proof}
 Since $E(\rho) = E(\rho^{ss})$ and $\tau(\rho) = \tau(\rho^{ss})$ (as $\rho$ may be deformed to $\rho^{ss}$ by definition), and since $\pm$-holomorphic representations are semisimple by definition, we will assume that $\rho$ is semisimple. Then, using the Higgs bundles formalism in equation \eqref{eqn:energyhiggs} and Lemma \ref{lemma:toledohiggs}, we have
 $$
 E(\rho) = \int_X \big\|\Phi\big\|^2 \frac{\omega_X^n}{n!}, \quad \tau(\rho) = \frac{1}{n} \int_X \Big( \big\|\Phi^+\big\|^2 - \big\| \Phi^-\big\|^2\Big) \frac{\omega_X^n}{n!},
 $$
 where $\Phi = \Phi^+ + \Phi^-$ is the decomposition into $\pm i$-eigenspaces of $Z$. Since this decomposition is orthogonal,
 $$
 E(\rho) = \int_X \Big( \big\|\Phi^+\big\|^2 + \big\| \Phi^-\big\|^2\Big) \frac{\omega_X^n}{n!},
 $$
 and the the conclusion follows.
\end{proof}

\begin{remark}\label{rmk:positivity}
 One can always assume that $\tau(\rho) \geq 0$. Indeed, the representations with negative $\tau$ are obtained by exchanging the complex structure on $Y$ (explicitely, for matrix groups, by $\rho \mapsto (\rho^t)^{-1}$). Accordingly, one would simply speak about ``holomorphic $\rho$'' instead of ``$\pm$-holomorphic $\rho$.
\end{remark}

\section{Royden's Ahlfors-Schwarz Lemma}\label{sec:royden}

The main technical instrument in our proof is the following theorem by Royden:

\begin{theorem}[\cite{Ro80}, Theorem 1]\label{thm:royden}
 Let $(\tX, g)$, $(Y, h)$ be Kähler manifolds, such that $\tX$ is complete and with Ricci curvature bounded from below by $k \leq 0$ and $Y$ has holomorphic sectional curvature bounded from above by $K < 0$. Then, for every holomorphic map $f \colon \tX \to Y$ of (holomorphic) rank $\leq \nu$, we have
 \begin{equation}\label{eqn:Royden}
 e(f) = \big\|\de f\big\|^2 \leq \frac{2 \nu}{\nu + 1} \frac{k}{K},
 \end{equation}
 where $e(f) = \|\de f\|^2$ is the energy density of $f$, given by the norm of $\de f$ as in Definition \ref{defn:energy}.
\end{theorem}

This theorem easily allows to prove the Milnor-Wood inequality for $\pm$-holomorphic representations (see Section \ref{sec:proof}). We will need, however, a closer inspection of the equality case. To do so, we adopt notations similar to Royden's: Let $\dev{}{z^\alpha}$ (resp. $\dev{}{w^j}$) be local normal coordinates on $\tX$ at a point $\tx_0$ (resp. $Y$ at $f(\tx_0)$), chosen in such a way that
$$
\dev{f^j}{z^\alpha} = \lambda_\alpha \delta_\alpha^j \dev{}{w^j}, \quad \lambda_\alpha \neq 0 \iff \alpha \leq \nu.
$$
With these notations, we have:
\begin{lemma}\label{lemma:roydenequality}
 Suppose that $f \colon \tX \to Y$ is as in Royden's Theorem \ref{thm:royden} and that equality holds generically in \eqref{eqn:Royden}. Then in fact it holds everywhere, and the rank of $\de f$ is constantly $\nu$. Then, using the coordinates chosen above, we have
 \begin{enumerate}
  \item $\lambda_1 = \dots = \lambda_\nu$, that is, $\big\|\dev{f}{z^1}\big\|^2 = \dots = \big\|\dev{f}{z^\nu}\big\|^2$; in particular, if $\nu = n$, then $f$ is a local isometry, up to a constant, hence, totally geodesic;
  \item For all $\alpha = 1, \dots, \nu$, $\dev{f}{z^\alpha}$ belongs to the subspace of $T^{1,0}Y$ where the maximum $K$ of the holomorphic sectional curvature is realized.
 \end{enumerate}
\end{lemma}
\begin{proof}
 First of all, remark that $e(f) = \|\de f\|^2$ is a continuous function, thus if equality holds generically in Theorem \ref{thm:royden} it holds everywhere. In particular, the rank cannot jump down, because of Royden's inequality.
 
 We now need to retrace Royden's proof to impose equality at all steps. The relevant ones for us are the following: For brevity of notations, denote by $S_{\alpha,\bar{\beta},\gamma,\bar{\delta}}$ the quantity $S\big(\dev{f}{z^\alpha}, \overline{\dev{f}{z^\beta}}, \dev{f}{z^\gamma}, \overline{\dev{f}{z^\delta}}\big)$, where $S$ is the Riemann curvature tensor of $Y$. Then:
 \begin{enumerate}[(i)]
  \item $\Delta \log e(f) \geq 2k - \frac{2}{e(f)} \sum_{\alpha,\gamma} S_{\alpha,\bar{\alpha},\gamma,\bar{\gamma}} \geq 2k - e(f) \frac{\nu+1}{\nu} K$ (\cite{Ro80}, Proposition 4);
  \item $\sum_\alpha S_{\alpha,\bar\alpha,\alpha,\bar\alpha} + 2 \sum_{\alpha \neq \gamma} S_{\alpha,\bar{\alpha},\gamma,\bar{\gamma}} \leq K \cdot e(f)^2$ (\cite{Ro80}, proof of the main Lemma);
  \item $S_{\alpha,\bar\alpha,\alpha,\bar\alpha} \leq K\|\dev{f}{z^\alpha}\|^4$ (hypothesis on the curvature).
 \end{enumerate}
 Remark also that the last inequality in (i) is obtained by summing (ii) to
 \begin{itemize}
  \item[(iii')] $\sum_\alpha S_{\alpha,\bar{\alpha},\alpha,\bar{\alpha}} \leq \frac{K}{\nu} e(f)^2$,
 \end{itemize}
 which is just the sum over all $\alpha$ of (iii) plus an application of the Cauchy-Schwarz inequality (recalling that $K \leq 0$).

 Now we impose $e(f) = 2 \frac{\nu}{\nu +1} \frac{k}{K}$. Then the first and last terms in (i) vanish; in particular, equality must hold everywhere. The conclusion is now straightforward, since by Cauchy-Schwarz $(\dev{f}{z^1},\dots,\dev{f}{z^\nu})$ must be a constant multiple of $(1, \dots, 1)$.
\end{proof}

\section{Proof of Theorem \ref{thm:main} and Corollary \ref{cor:standard}}\label{sec:proof}

We will split the proof into several intermediate results.

\begin{prop}\label{prop:MWinequality}
 Let $X$ be a compact Kähler manifold of dimension $n > 1$, $\Gamma = \pi_1(X)$ its fundamental group, $\rho \colon \Gamma \to G$ a representation to a Lie group of non-compact Hermitian type. Suppose that $\rho$ is deformable to a $\pm$-holomorphic representation. Then it satisfies a Milnor-Wood inequality:
 \begin{equation}\label{eqn:MWinequality}
 |\tau(\rho)| \leq \tau_{\max} = \frac{-2k}{n+1} \rk(G) \Vol(X).
 \end{equation}
\end{prop}
\begin{proof}
 Since the Toledo invariant does not change under deformation, we can assume that $\rho$ is $\pm$-holomorphic itself (or, indeed, holomorphic, see Remark \ref{rmk:positivity}). Then, just combine Royden's inequality with Lemma \ref{lemma:toledohiggs}, to obtain:
 $$
 \big|\tau(\rho)\big| \leq \frac{1}{n} E(\rho) = \frac{1}{m} \int_X e(f) \frac{\omega_X^n}{n!} \leq \frac{-2k}{n} \rk(G) \frac{\nu}{\nu +1} \Vol(X).
 $$
 Since $\nu \leq n$ implies $\frac{\nu}{\nu+1} \leq \frac{n}{n+1}$, the result is proved.
\end{proof}

From now on, we shall focus on the study of the case of equality, that is, maximal representations. Since every semisimple group $G$ splits as product of almost simple pieces, which are all Hermitian if and only if $G$ is, and since the Toledo invariant is additive under such a splitting, we shall suppose from now on that $G$ is simple. Also, thanks to Remark \ref{rmk:positivity}, we will simply speak about holomorphic maps or representations. Thanks to our discussion on the case of equality in Royden's Theorem, we immediately have:

\begin{lemma}\label{lemma:maximalholomorphic}
 Let $\rho \colon \Gamma \to G$ be a maximal holomorphic representation of a Kähler group to a non-compact Lie group of Hermitian type $G$. Then the holomorphic $\rho$-equivariant map $f \colon \tX \to Y$ gives an isometric and biholomorphic embedding (up to rescaling the metric on $\tX$) with its image.
\end{lemma}
\begin{proof}
 From the proof of Proposition \ref{prop:MWinequality} it is clear that if equality holds then $\nu = n$, that is, $f$ is a generic immersion. As in the proof of Lemma \ref{lemma:roydenequality} one passes easily from generic immersion to a genuine immersion. Applying this same lemma, then, $f$ is proven to be a local isometry with its image (possibly after rescaling the metric on $\tX$). But then, $Y$ is uniquely geodesic, thus a local isometry must be injective, as well. An injective holomorphic immersion is a biholomorphism with the image, concluding the proof.
\end{proof}

We need to fix some notation for the following. Let $\rho \colon \Gamma \to \SU(p,q)$ be a semisimple representation and $(\calE, \Phi)$ be the Higgs bundles associated to this representation and some harmonic metric $f$. Split $\Phi = \Phi^+ + \Phi^-$, according to the decomposition into $\pm i$-eigenspaces of the almost complex structure $Z$ on $Y_{p,q}=\SU(p,q)/S(U(p) \times U(q))$. Then there is a holomorphic splitting $\calE = V \oplus W$, such that $\Phi^+$ maps $W$ to $V \otimes \Omega_X^1$ and $\Phi^-$ maps $V$ to $W \otimes \Omega_X^1$. Write $\beta \colon W \otimes T^{1,0} X \to V$ for the composition of $\Phi^+$ with the contraction of holomorphic vector fields with holomorphic 1-forms (and $\gamma \colon V \otimes T^{1,0}X \to W$ for the analogous construction with $\Phi^-$). We have the following Lie algebra result:

\begin{lemma}\label{lemma:matrices}
 Let $G$ be a classical Lie group of non-compact Hermitian type and $Y$ its associated symmetric space, which we normalize so as to have holomorphic sectional curvature pinched between $-1$ and $-\frac{1}{\rk(G)}$. Suppose that at a point $y \in Y$ there is a complex $n$-subspace of the holomorphic tangent bundle $\calT_yY$, with $n \geq 2$, entirely contained in the locus $\calL$ of maximal holomorphic sectional curvature $-\frac{1}{\rk(G)}$. Then necessarily $G = \SU(p,q)$, with $p \geq nq$ (or $q \geq np$). If $f \colon \tX \to Y_{p,q}$ is an immersion such that $\de f(T^{1,0}\tX) \subseteq \calL$, then $\beta$ is injective.
\end{lemma}
The proof of this lemma is a rather long matrix computation and will be postponed until the next section. As an immediate consequence of Lemmas \ref{lemma:maximalholomorphic} and \ref{lemma:matrices}, if $\rho$ is maximal and (deformable to) holomorphic, then $G$ is necessarily $\SU(p,q)$, so from now on we will stick to this situation, and assume $p \geq q$.

We now state the main point that allows us to extend our results from a holomorphic representation $\rho$ to all the other representations in the same connected component of $\M$. This is the crucial point where the hypothesis $n > 1$ is used.

\begin{lemma}\label{lemma:holomorphiccomponent}
 Suppose that $\Gamma = \pi_1(X, x)$ is the fundamental group of a compact Kähler manifold of dimension $n > 1$. Suppose that $\rho \colon \Gamma \to \SU(p,q)$ is a maximal $\pm$-holomorphic representation. Then every semisimple representation $\rho'$ in the same connected component as $\rho$ is $\pm$-holomorphic, as well.
\end{lemma}
\begin{proof}
 This follows from a crucial remark in \cite{KoMa10}, which is as follows:
 \begin{equation}\label{eqn:KM}
 \text{If } \beta \colon W \otimes T^{1,0}X \to V \text{ is injective, } n \geq 2 \implies f \text{ is holomorphic, i.e. } \gamma = 0.
 \end{equation}
 They prove this remark by taking, at a given point $x \in X$, two linearly independent tangent vectors $\xi, \eta \in T_x^{1,0}X$, and exploiting the Higgs bundles relation $\Phi \wedge \Phi = 0$, that is, for every $v \in V_x$, $\beta(\xi) \gamma(\eta) v = \beta(\eta) \gamma(\xi) v$. Then by injectivity of $\beta$ one may conclude that $\gamma$ vanishes at $x$. This fact, together with Lemma \ref{lemma:matrices}, gives in fact an equivalence of the two notions of holomorphicity of $f$ and injectiveness of $\beta$ (for maximal $\rho$).
 
 Now remark that the condition of $\beta$ being injective is open in the set of semisimple representations, since the maps associating a Higgs bundle $(\calE, \Phi)$ to a semisimple representation is continuous and $\beta$ not being injective is a minor vanishing condition on $\Phi$. On the other hand, by Proposition \ref{prop:energyinequality}, under the hypothesis $\tau(\rho) = \tau_{\max} = \frac{q}{n} \Vol(X)$, a representation $\rho$ is holomorphic $\iff E(\rho) = \tau_{\max}$. Since the energy is a proper map on $\M$ (this is a standard application of Uhlenbeck's compactness criterion; see, for example, \cite{DaDoWe98}, Proposition 2.1), the conjugacy classes of holomorphic representations form a compact subset thereof, hence the preimage in $\Hom(\Gamma, \SU(p,q))^{ss}$ is a closed subset. Being open, as well, it must consist of connected components.
\end{proof}

We now want to exclude non-semisimple representations. For $n \geq 2$, this can be proved explicitly as follows: Let $\rho$ be a non-semisimple representation with $E(\rho) = |\tau(\rho)|$ (so that, in particular, its semisimplification $\rho^{ss}$ is holomorphic, and inequality \eqref{eqn:MWinequality} holds). Suppose that $\rho$ is maximal, and denote by $f$ a $\rho^{ss}$-equivariant holomorphic map and by $y_\infty$ a point at infinity in $Y_{p,q}$ fixed by $\rho(\Gamma)$. One can see that if $\tx_0$ is a base point, denoting by $o = f(\tx_0)$ and letting $\chi \in \g$ represent a vector pointing from $o$ to $y_{\infty}$, then $[\de f_{\tx_0}(\xi), \chi] = 0$ for every $\xi \in T_{\tx_0}\tX$. But it is easy to see from the proof of Lemma \ref{lemma:matrices} that the special form $\de f_{\tx_0}(\xi)$ must have implies that no $\chi$ can centralize both $\de f_{\tx_0}(\xi)$ and $\de f_{\tx_0}(\eta)$ if $\xi$ and $\eta$ are any two linearly independent vectors in $T_{\tx_0}\tX$. However, there is a shorter proof that gives a stronger result (and that works in the $n=1$ case, as well). This is essentially due to Burger--Iozzi \cite{BuIo07}; I thank Beatrice Pozzetti for pointing this out to me.

\begin{lemma}
 Let $X$ be any Kähler manifold, $\rho \colon \Gamma = \pi_1(X) \to G$ a representation to a Lie group of Hermitian type. Suppose that the Milnor--Wood inequality \eqref{eqn:MWinequality} holds for $\rho$ (for example, that $\tX = \B^n$ or that $\rho$ is deformable to a $\pm$-holomorphic representation). Then, if $\rho$ is maximal, the Zariski closure $G_0 = \overline{\rho(\Gamma)}$ is a Lie group of Hermitian symmetric type. In particular, $\rho$ is semisimple.
\end{lemma}
\begin{proof}
 Let $K_0 = G_0 \cap K$ be a maximal compact subgroup of $G_0$, and write $Y_0 = G_0/K_0 = Y_0^{(1)}\times\dots\times Y_0^{(k)}$ for the decomposition into irreducible pieces. Then the Milnor--Wood inequality also holds for the projection of $\rho$ to any of the groups $\Isom(Y_0^{(i)})$, and $\rho$ is maximal if and only if each of these projections is. We may thus assume that $Y_0$ is irreducible itself.
 
 As in \cite{BuIo07}, \S 5, one has:
 $$
 H_{cb}^2(G_0, \R) \cong H_c^2(G_0, \R) \cong H^2(\calA^\bullet(Y_0)^{G_0}),
 $$
 where the first is the isomorphism between the bounded cohomology of $G_0$ and its continuous cohomology, and the last is the van Est isomorphism relating these spaces to the de Rham cohomology of $G_0$-invariant differential forms. Then, $H_{cb}^2(G_0, \R)$ is either isomorphic to $\R$ if $G_0$ is of Hermitian type, or $0$ otherwise. Because of this, denoting by $i \colon Y_0 \hookrightarrow Y$, we have that $i^*\omega_Y$ vanishes in cohomology if $G_0$ is not Hermitian, i.e. $i^*\omega_Y = \de \eta$ for some $G_0$-invariant 1-form $\eta$ on $Y_0$. Letting $f_0 \colon \tX \to Y_0$ be any $\rho$-equivariant smooth map, we have
 $$
 \tau(\rho) = \frac{1}{n!} \int_X f_0^*i^*\omega_Y \wedge \omega_X^{n-1} = \frac{1}{n!} \int_X \de f_0^*\eta \wedge \omega_X^{n-1} = 0,
 $$
 using $\de \omega_X = 0$ and Stokes theorem.
\end{proof}

Finally, let us explain how Corollary \ref{cor:standard} follows from Theorem \ref{thm:main} and \cite{Ha12}.

\begin{proof}[Proof of Corollary \ref{cor:standard}]
 Under the hypothesis that $\rho$ extends to $\SU(n,1)$, the equivariant harmonic map is just the totally geodesic map on the associated symmetric spaces $f \colon \B^n \to Y$. Hamlet \cite{Ha12} proves that every totally geodesic \emph{tight} map is (anti)-holomorphic. Thus we need only to see that a maximal representation factoring through $\SU(n,1)$ necessarily induces a tight map. The argument for doing that is standard and is taken essentially from \cite{BuIoWi10}.
 
 Recall that by the Van Est isomorphism, the bounded Kähler class $\kappa_{\B^n}^b$ essentially coincides with the Kähler form $\omega_X$ (and similarly on $Y$). Since $\mathbb{H}_{cb}^2(\SU(n,1), \R) \cong \R$, there is a $\lambda$ such that $f^*\kappa_{Y}^b = \lambda \kappa_{\B^n}^b$. One sees easily as in \cite{BuIoWi10} that $f$ is tight if, and only if, $\lambda = \rk(G)$. But then the Toledo invariant is
 $$
 \tau(\rho) = \frac{1}{n!} \int_X \lambda \omega_X^n = \lambda \Vol(X).
 $$
 Using Burger--Iozzi Theorem \eqref{eqn:MWBI}, it follows that any $\rho$ factoring through $\SU(n,1)$ is maximal if and only if $f$ is tight.
\end{proof}

\section{Proof of Lemma \ref{lemma:matrices}}\label{sec:prooflemma}

In order to prove Lemma \ref{lemma:matrices}, we need to recall briefly the structure of the simple Lie groups of Hermitian type and the definition of Hermitian sectional curvature.

\begin{defn}
 Let $Y$ be a Kähler manifold, and, in a local holomorphic frame $\{\dev{}{z^i}\}$, write $g_{i\bar{j}}$ for the metric tensor and $R_{i\bar{j}k\bar{\ell}}$ for the Riemann curvature tensor. The Hermitian sectional curvature along a holomorphic tangent vector $\xi = \xi^j \dev{}{z^j} \in T^{1,0}Y$ is defined by
 $$
 K_H(\xi, \xi) = \frac{R_{i\bar{j}k\bar{\ell}} \xi^i \bar\xi^j \xi^k \bar\xi^\ell}{\big(g_{i\bar{j}} \xi^i\bar\xi^j\big)^2}
 $$
\end{defn}

In the case where $Y$ is a Hermitian symmetric space associated to a matrix Lie group, we can endow it with the invariant metric
$$
g(A, B) = \trace\big(AB^*\big)
$$
(beware that in general, this differs from both the one induced by the Killing form and the one having minimal holomorphic sectional curvature $-1$ by some constants, but we will stick to this definition for the ease of computations). With this form, the holomorphic sectional curvature of a classical Hermitian symmetric space is given by
$$
 K_H(M, M) = -\frac{\trace\big([M, M^*]^2\big)}{\big(\trace(MM^*)\big)^2}, \quad M \in \p^{1,0} \subset \g^\C.
$$
We will now distinguish between the four different classes of the classical Lie group of Hermitian type to identify the locus maximizing the holomorphic sectional curvature (the bounds on this quantity are classic, but for completeness they will be proved in Lemma \ref{lemma:curvatureequality}, as well).
\begin{itemize}
 \item $G = \SU(p,q)$: In this case, $\p^{1,0}$ consists of matrices of the form $M = \begin{pmatrix}0&A\\0&0\end{pmatrix}$, with $A$ a complex valued $p \times q$ matrix. In particular, $\|M\|^2 = \trace(M^*M) = \trace(A^*A)$ and $\trace\big([M, M^*]^2\big) = 2 \trace\big((A^*A)^2\big)$. Thus the holomorphic sectional curvature along $M$ satisfies:
$$
-2 \leq K_H(M, M^*) \leq -\frac{2}{\min(p,q)}.
$$
 \item $G = \Sp(2n, \R)$: A matrix $M$ is in $\p^{1,0}$ if it is of the form $\begin{pmatrix}A & iA\\iA & -A\end{pmatrix}$, with $A$ an $n \times n$ complex matrix such that $A^t = A$. In particular, $\trace\big([M, M^*]^2\big) = 32 \trace\big((A^*A)^2\big)$ and $\trace(M^*M) = 4\trace(A^*A)$. This implies
$$
-2 \leq K_H(M,M^*) \leq -\frac{2}{n}.
$$
 \item $G = \SO(p,2)$: Here, $\p^{1,0}$ is made of matrices $M = \begin{pmatrix}0&A\\A^t&0\end{pmatrix}$ where $A$ is a $p \times 2$ matrix whose two vector columns are of the form $v$ and $iv$ for some vector $v \in \C^p$. Then we have $\trace\big([M, M^*]^2\big) = 16 \|v\|^4 - 8 |\scal{v}{\bar{v}}|^2$ and $\trace(M^*M) = 4\|v\|^2$. Thus
$$
-1 \leq K_H(M, M^*) \leq -\frac{1}{2}.
$$
 \item $G = \SO^*(2n)$: In this case $\p^{1,0}$ contains the matrices $M$ of the form $\begin{pmatrix}iA&-A\\-A&-iA\end{pmatrix}$, with $A$ a complex $n \times n$ matrix such that $A^t = -A$. Similar computations as in the case of $\Sp(2n, \R)$ show that $\trace\big([M, M^*]^2\big) = 32 \trace\big((A^*A)^2\big)$ and that $\trace(M^*M) = 4\trace(A^*A)$. However, in this case $A$ is skew-symmetric, and one has the stricter inequalities:
 $$
 -1 \leq K_H(M, M^*) \leq -\frac{1}{\lfloor n/2\rfloor}.
 $$
\end{itemize}
With these preliminaries, Lemma \ref{lemma:matrices} is deduced from the following:
\begin{lemma}\label{lemma:curvatureequality}
 \begin{enumerate}
  \item Let $A$ be a non-zero $p\times q$ complex-valued matrix with $p \geq q$. Then
 \begin{equation}\label{eqn:pqinequality}
 \frac{1}{q} \leq \frac{\trace\big((A^*A)^2\big)}{\big(\trace(A^*A)\big)^2} \leq 1.
 \end{equation}
 The upper bound is reached by all matrices of rank $1$, the lower one by those satisfying $A^*A = \lambda I_q$, for some $\lambda > 0$. The maximal dimension of a linear space $L \subset M_{p\times q}(\C)$ such that every $A \in L$ realizes this minimum is $\big\lfloor \frac{p}{q} \big\rfloor$. More precisely, if $A_1, \dots, A_k$ are linearly independent matrices in such a $L$, then necessarily $(A_1,\dots,A_k)$ gives an immersion $\C^{kq} \hookrightarrow \C^p$.
 \item Let $v \in \C^p$ be a vector with $p \geq 2$. Then
 $$
 0 \leq \|v\|^4 - \big|\bigscal{v}{\bar{v}}\big|^2 \leq \|v\|^4.
 $$
 The upper bound is reached by vectors $v = (z_1,\dots,z_p)$ such that $\sum z_j^2 = 0$, and the lower bound by those such that $v \in U(1) \cdot \R^n$. The maximal dimension of a $\C$-linear subspace $L \subset \C^p$ such that every $v \in L$ realizes the minimum is $1$.
 \item Let $A$ be a non-zero $n \times n$ complex-valued skew-symmetric matrix (that is, $A^t = -A$). Then
 \begin{equation}\label{eqn:skewinequality}
 \frac{1}{2\lfloor\frac{n}{2}\rfloor} \leq \frac{\trace\big((A^*A)^2\big)}{\big(\trace(A^*A)\big)^2} \leq \frac{1}{2}.
 \end{equation}
 The upper bound is reached by rank $2$ matrices; if $n$ is even, the lower bound is realized by the matrices satisfying $A^*A = \lambda I_n$. If $n$ is odd by those such that $A^*A$ has a (strictly) positive eigenvalue of multiplicity $n-1$. Finally, the maximal dimension of a linear space $L$ such that every $A \in L$ realizes this minimum is $1$ for $n \geq 4$.
 \end{enumerate}
\end{lemma}
\begin{proof}
 1. The inequalities are classical: $A^*A$ being Hermitian and positive definite, it has positive eigenvalues $\lambda_1, \dots, \lambda_q$. We are then stating that
  $$
  \frac{1}{q} \Big(\sum_j \lambda_j\Big)^2 \leq \sum_j \lambda_j^2 \leq \Big(\sum_j \lambda_j\Big)^2,
  $$
  which is clear. The equality on the right hand side holds if there is only one non-zero $\lambda_1$, and the one on the left if all the $\lambda_j = \lambda$ are equal, hence $A^*A = \lambda I_q$. Let $m = \lfloor \frac{p}{q} \rfloor$. Define $A_1, \dots, A_{m}$ as follows:
  \begin{equation}\label{eqn:Aistd}
  A_1 = \begin{pmatrix}I_q\\0_q\\0_q\\\vdots\end{pmatrix},\quad A_2 = \begin{pmatrix}0_q\\I_q\\0_q\\\vdots\end{pmatrix}, \quad A_3 = \begin{pmatrix}0_q\\0_q\\I_q\\\vdots\end{pmatrix},\dots
  \end{equation}
  where $0_q$ is the zero $q \times q$ matrix (and, if $p > mq$, each matrix $A_i$ is completed with zeros, as well). Then it is clear that for every $t_0, \dots, t_{m-1}$ we have
  $$
  (t_1 A_1 + \dots + t_m A_m)^*(t_1 A_1 + \dots + t_m A_m) = \sum_j |t_j|^2 I_q.
  $$
  We need to prove that there can be no $(m+1)$-dimensional complex linear space with this property. To that aim, we prove that any $k$-tuple of linearly independent matrices $A_1, \dots, A_k$ such that every non-zero linear combination $\sum_j t_j A_j$ realizes the minimum in \eqref{eqn:pqinequality} can be modified to another such $k$-tuple so that, moreover, every pair of column vector of any of these matrices are mutually orthogonal. This will clearly imply that $kq \leq p$ and also the claim about the injectivity of $(A_1,\dots,A_k)$. For brevity's sake, we will only consider two matrices $A$, $B$ such that $(tA + sB)^*(tA+sB) = \lambda_{t,s} I_q$ and modify them to such an ``orthonormal pair'', leaving the straightforward adaptations for the general case to the reader. First of all, rescale them to get $A^*A = B^*B = I_q$, so that, letting $v_1, \dots, v_q$ be the column vectors of $A$, $w_1,\dots,w_q$ those of $B$, each set consists of $q$ orthonormal vectors. Write $w_j = \sum_{i=1}^q a_{ij} v_i + w'$, with $w'$ orthogonal to each of the $v_i$'s. By hypothesis, there exist positive reals $\lambda, \mu$ such that $(A+B)^*(A+B) = \lambda I_q$ and $(A+iB)^*(A+iB) = \mu I_q$. Writing down these conditions explicitly, one finds
  \begin{equation}\label{eqn:aij}
  a_{ij} = \Big(\frac{\lambda-2}{2} + i \cdot \frac{2-\mu}{2}\Big) \delta_{ij} \quad 1 \leq i,j \leq q.
  \end{equation}
  Thus, either $B - \frac{1}{2} \big(\lambda-2 + i (2-\mu)\big) A$ has the required property, or it is zero, forcing $A$ and $B$ to be linearly dependent in the first place.
  
  2. One can repeat computations similar to the ones above; however, one can easily be more precise: The locus realizing the minimum $\|v\|^4 = |\scal{v}{\bar v}|^2$ consist of the vectors $v$ such that $v = e^{i\theta} \bar{v}$. All real vectors $v \in \R^p$ are in this locus, and in fact if $v \in \C^p$ is in there, then $e^{-i\theta/2}v$ is real. Thus $L \subset U(1) \cdot \R^p$ has at most dimension $1$.
  
  3. In this case, we make use of the Youla decomposition of complex skew-sym\-me\-tric matrices (cfr. \cite{Yo61}) to infer that, up to multiplying by a unitary matrix $U$ on the left and by $U^t$ on the right, any skew-symmetric matrix $A$ can be reduced to a block diagonal matrix with only $2 \times 2$ skew-symmetric blocks with real entries (plus one zero if $n$ is odd). This forces the non-zero eigenvalues of $A^*A$ to appear in pairs. Writing $\lambda_1, \lambda_1, \dots, \lambda_k, \lambda_k$ for the non-zero eigenvalues of $A^*A$, so that $2k \leq n$, \eqref{eqn:skewinequality} follows from the obvious
  $$
  \frac{1}{2k} \Big(\sum_j 2\lambda_j\Big)^2 \leq 2 \sum_j \lambda_j^2 \leq \frac{1}{2} \Big(\sum_j 2\lambda_j\Big)^2.
  $$
  Equalities hold if there is only one pair $k = 1$ or the maximal possible number $k = \lfloor \frac{n}{2} \rfloor$. When $n$ is even, the latter reduces to $A^*A = \lambda I_n$, so the assertion about the maximal dimension of $L$ follows from the first case with $p = q = n$. When $n$ is odd, however, we can only infer that $A^*A$ is unitary conjugate to $\diag(\lambda, \dots, \lambda, 0)$, so some other argument is needed.
  
  Suppose first that $n \geq 9$. By contradiction, suppose that $A$, $B$ are $n \times n$ skew-symmetric matrices with $n \geq 9$ odd, such that for every $t, s \in \C$ the matrix $(tA + sB)^*(tA+sB)$ has one positive eigenvalue of multiplicity $n-1$ (which again we take to be $1$ for $(t, s) = (1,0)$ or $(0,1)$). In particular, for every $t,s$ there is a subspace $V_{s,t} \subset \C^n$ of codimension $1$ where $(tA + sB)^*(tA+sB)|_{V_{s,t}} = \lambda I_{V_{s,t}}$. Let $W$ be the intersection $V_{1,0} \cap V_{0,1} \cap V_{1,1} \cap V_{1,i}$, which has at most codimension $4$. Extend an orthonormal basis of $W$ to one of $\C^n$, and write $v_1,\dots, v_n$ (resp. $w_1,\dots,w_n$) for the column vectors of $A$ (resp. $B$). By assumption, $v_1, \dots, v_{n-4}$ are orthonormal, and the same is true for $w_1, \dots, w_{n-4}$. We can repeat the steps as in \eqref{eqn:aij} to modify $A$ and $B$; since we cannot get to $v_1,\dots,v_{n-4},w_1,\dots,w_{n-4}$ forming a $(2n-8)$-tuple of orthogonal vectors, this forces $w_1 = e^{i\theta}v_1, \dots, w_{n-4} = e^{i\theta}v_{n-4}$. But then $e^{i\theta}A-B$ has at most rank $4$, hence so does $(e^{i\theta}A-B)^*(e^{i\theta}A-B)$; since this can only have rank $n-1$ or $0$, it must be $0$, hence $e^{i\theta}A = B$, and we are done.
  
  The lower dimensional cases must be treated separately. Suppose for example that $n = 5$. The locus of maximal holomorphic sectional curvature is given by
  $$
  V = \Big\{ A \in SS_5(\C) \cong \C^{10} \ : \ F(A) := \big(\trace(A^*A)\big)^2 - 4 \trace\big((A^*A)^2\big) =0 \Big\}.
  $$
  The isomorphism between the $5\times5$ skew symmetric complex matrices and $\C^{10}$ can be given, for example, by the following ordering:
  \begin{equation}\label{eqn:paramSS}
  SS_5(\C) \ni A = \begin{pmatrix}
                    0 & a_1 & a_2 & a_3 & a_4\\
                    -a_1 & 0 & a_5 & a_6 & a_7\\
                    -a_2 & -a_5 & 0 & a_8 & a_9\\
                    -a_3 & -a_6 & -a_8 & 0 & a_{10}\\
                    -a_4 & -a_7 & -a_9 & -a_{10} & 0
                   \end{pmatrix}.
  \end{equation}
  Remark that $U \in \SU(5)$ acts on $A \in V$ by $A \mapsto U^t A U$; this action is linear on $V$, so it descends to an action on the set of complex lines $\mathbb{P}V$. By Youla's decomposition, this action is transitive. To prove that no complex plane is contained in $V$, we can thus choose work locally around any preferred point. Our choice will be $A_0$ whose parametrization as in \eqref{eqn:paramSS} has $a_1 = a_8 = 1$ and 0 elsewhere. Since $V$ is defined by $F = 0$, any direction of a plane $L \subset V$ through $A_0$ must make the Levi form vanish, i.e. $(\partial \bar\partial F)(\xi) = 0$ for all $\xi \in T_{A_0}L$. Let us compute the quadratic form $Q = \partial \bar\partial F$. Since $\trace(A^*A) = 2\sum_j |a_j|^2$,
  \begin{align*}
  \dev{}{a_k} \dev{}{\bar{a_j}} \big( \trace(A^*A)\big)^2 &= \begin{cases}
                                                                     8 \bar{a_k}a_j            & \text{if } j \neq k\\
                                                                     8 |a_j|^2 + 4\trace(A^*A) & \text{if } j = k,
                                                                    \end{cases}\\
  \dev{}{a_k} \dev{}{\bar{a_j}} \big( \trace(A^*AA^*A)\big) &= -16 \trace\Big(\dev{A}{\bar{a_j}} \dev{A}{a_k} A^*A\Big),
  \end{align*}
  where in the last equality we have used that all the matrices involved are skew-symmetric. Specializing at $A = A_0$, the first line gives $8$ for $j \neq k \in \{1, 8\}$, and 0 for other $j \neq k$, 24 for $j = k \in \{1, 8\}$ and 16 otherwise. The second line give $0$ whenever $j \neq k$, 16 for $j = k \in \{4, 7, 9, 10\}$ and 32 otherwise: Indeed, $A_0^*A_0 = \diag(1,1,1,1,0)$ and $\partial{A}{a_j} \partial{A}{a_k}$ has no non-zero terms on the diagonal unless $j = k$. In this case, it has exactly two $-1$ on the diagonal, of which one is in position $(5,5)$ if and only if $j = k \in \{4,7,9,10\}$.
  
  The quadratic form $Q$ is the difference of the two quadratic forms above, and a straightforward computation proves that it is negative semi-definite with 5-dimensional kernel generated by $\langle e_4, e_7, e_9, e_{10}, e_1+e_8\rangle$, where $e_i$ corresponds to the $i$-th element of the canonical basis of $\C^{10}$. The direction $e_1 + e_8$ is the ``trivial one'', along multiples of $A_0$. It is thus enough to prove that there is no other line passing by $A_0$ on which $F$ vanishes. This follows from the straightforward computation giving: For all $a,b,c,d \in \C$,
  $$
  F(A_0 + ae_4 + be_7 + ce_9 + de_{10}) = -4\big(|a|^2 + |b|^2 + |c|^2 + |d|^2\big).
  $$
  The case $n = 7$ is essentially identical: in that case, $V$ is defined by the two equations
  \begin{align*}
  V = \Big\{ A \in SS_7(C) \cong \C^{21} \ : \ F(A) = \big(\trace(A^*A)\big)^2 - 6 \trace\big((A^*A)^2\big) & =0,\\
                                                     \big(\trace(A^*A)\big)^3 - 36 \trace\big((A^*A)^3\big) & =0\Big\},
  \end{align*}
  but the same proof gives that no complex plane is contained in the bigger subspace defined by $F = 0$ only: $Q$ is still negative semi-definite, with $7$-dimensional kernel composed by 1 ``trivial'' direction as above, plus 6 more along which $F < 0$. The same proof should also work in higher dimension, but we preferred to give a different, shorter one to avoid introducing heavy notations.
\end{proof}

\section{Proof of Theorem \ref{thm:classification}}\label{sec:proofclass}
The fact that for the stated $G$'s every minimum of the Morse function (in a connected component) is holomorphic (or Fuchsian for $G = \Sp(2n, \R)$) is due to Bradlow, García-Prada and Gothen (see \cite{BrGPGo06} and the references therein, where also the numbers of maximal connected components are computed). We now discuss the classification of maximal holomorphic representations $\rho \colon \Gamma_g \to G$. Let $f \colon \tilde \Sigma \cong \H_\R^2 \to Y = G/K$ be a $\rho$-equivariant holomorphic map. Fix a base point $\tx_0 \in \tilde \Sigma$; up to conjugation of $\rho$, we can suppose that $f(\tx_0) = eK$. Thanks to Lemma \ref{lemma:roydenequality}, $f$ is totally geodesic, so there exists a representation $\rho_{tot} \colon \SL_2(\R) \to G$ such that $f(\tx) = f(g \tx_0) = \rho_{tot}(g)\cdot K$. It is then easy to see that $\rho \colon \Gamma \to G$ must be of the form $\rho(\gamma) = \chi(\gamma) \rho_{tot}(\gamma)$, where $\chi \colon \Gamma_g \to Z_G(\rho_{tot}(\SL_2(\R))) \subset K$ takes values in the centralizer of the image of $\rho_{tot}$, so that
$$
\forall \tx = g\tx_0 \in \tilde \Sigma, \quad \rho(\gamma)f(\tx) = \chi(\gamma) \rho_{tot}(\gamma g) K = \rho_{tot}(\gamma g) \cdot K = f(\gamma \tx).
$$
To conclude the proof, we will do the following in each of the possible cases for $G$:
\begin{enumerate}
 \item Describe, thanks to Lemma \ref{lemma:curvatureequality}, the possible maps $f_* \colon T_{\tx_0}\tilde \Sigma \to T_{eK}Y$, corresponding to some Lie algebra homomorphism $f_* \colon \sl_2(\R) \to \g$.
 \item Choose a preferred element between these possibilities, and compute the corresponding Lie group homomorphism, as in Table \ref{tab:representations}.
 \item Describe the centralizer of the image of these homomorphism, in order to complete table \ref{tab:centralizers}.
 \item To rule out the remaining ambiguities, we have to check that the maximal compact subgroup $K$ acts transitively on the possible choices compatible with 1.
\end{enumerate}

$\bullet \quad G = \SU(p,q)$, $p \geq q$: In this case, the holomorphic tangent bundle $\p^{1,0}$ must be sent to something of the form $\begin{pmatrix}0&A\\0&0\end{pmatrix}$, where $A$ is a $p \times q$ matrix such that $A^*A = \lambda I_q$. In this case, it is clear that $K = S(U(p)\times U(q))$ acts transitively on the possible choices by $A \mapsto PAQ^*$. One preferred choice is $A = \begin{pmatrix} I_q\\0\end{pmatrix}$. In that case, the Lie algebra morphism is clearly the one shown in Table \ref{tab:representations}, which is induced by the linear Lie groups homomorphism $\SU(1,1) \to \SU(p,q)$ shown in the same table. A straightforward check gives that the centralizer of its image is given by
$$
Z = Z_G(f_*(\sl_2(\R))) = \Bigg\{ \begin{pmatrix}U\\&F\\&&U\end{pmatrix}, \  U \in U(q), F \in U(p-q), \det(U)^2 \det F = 1 \Bigg\}.
$$
When $p = q$, there is no $F$, hence this group is isomorphic to $\SU(q) \rtimes \Z/2\Z$. When $p > q$, the determinant of $F$ is uniquely determined, so $Z$ is isomorphic to $U(q) \times \SU(p-q)$. Because of that, and since $\Hom(\Gamma_g, \GL(q,\C))$ and hence $\Hom(\Gamma_g, U(q))$ are connected, $\Hom(\Gamma_g, Z)$ has $2^{2g}$ components in the former case, and it is connected in the latter.

$\bullet \quad G = \Sp(2n,\R)$, $n \geq 3$: Again by Lemma \ref{lemma:curvatureequality}, $\p^{1,0} \mapsto W = \Big\{\begin{pmatrix}A&iA\\iA&-A\end{pmatrix}\Big\}$, where $A$ is complex symmetric such that $A^*A = \bar{A}A = \lambda I_n$, for some $\lambda > 0$. As a preferred choice, we can take $A = I_n$. This corresponds to the Lie algebra morphism as in Table \ref{tab:representations}, which again corresponds trivially to a linear Lie group homomorphism $\SL_2(\R) \to \Sp(2n,\R)$. The centralizer of the image is readily computed to be of the form $\bigg\{\begin{pmatrix}Q\\&Q\end{pmatrix},\ Q \in O(n)\bigg\}$. Hence, to see that $K$ acts transitively on the set of complex lines in $W$, and since the Lie algebra is
$$
\k = \bigg\{\begin{pmatrix}B&C\\-C&A\end{pmatrix},\ B^t=-B, C^t=C\bigg\},
$$
it is enough to consider adjunction by elements of the form $\exp k$ where $k = \begin{pmatrix}0&C\\-C&0\end{pmatrix}$. The usual formula $\Ad_{\exp} = e^{\ad}$ gives
$$
\Ad_{\exp k}\begin{pmatrix}
             I & iI\\
	     iI & -I
            \end{pmatrix}
         =
            \begin{pmatrix}
             \exp(2iB)&i\exp(2iB)\\
             i\exp(2iB)&-\exp(2iB)
            \end{pmatrix}.
$$
When $B$ varies across all symmetric real matrices, $\exp(2iB)$ gives all the unitary symmetric ones, that proves the transitivity. Finally, the count of the connected components follows from $\big|\Hom\big(\Gamma_g, O(n)\big)\big| = 2^{2g+1}$.

$\bullet \quad G = \SO_0(n,2)$, $n \geq 4$: In this case Lemma \ref{lemma:curvatureequality} gives $\p^{1,0} \mapsto \bigg\{\begin{pmatrix}0_n&v&iv\\v^t&0&0\\iv^t&0&0\end{pmatrix}\bigg\}$, where $v \in U(1) \cdot \R^n$. It is clear in this case that the maximal compact acts transitively on the complex lines therein, since $O(n)$ does on the real lines in $\R^n$. A preferred choice is $v = e_1$, the first vector of the canonical base of $\R^n$. This corresponds to the Lie algebra map $f_*$ in Table \ref{tab:representations}. A long but easy computation using $\exp \circ f_* = \rho_{tot} \circ \exp$ gives the Lie group homomorphism
$$
\begin{pmatrix}
 \alpha&\beta\\
 \bar\beta&\bar\alpha
\end{pmatrix}
 \textrightarrow{\rho_{tot}}
\begin{pmatrix}
 2|\beta|^2+1 & 0 & \cdots & 0 & 2 \Re(\alpha\bar\beta) & 2 \Im(\alpha \bar\beta)\\
 0 & 1 & 0 & \vdots & 0  & 0\\
 \vdots & 0 & \ddots & 0 & \vdots & \vdots\\
 0 & \cdots & 0 & 1 & 0 & 0\\
 2 \Re(\alpha\beta) & 0 & \cdots & 0  & \Re(\alpha^2 + \beta^2) & \Im(\alpha^2 + \beta^2)\\
 -2 \Im(\alpha\beta) & 0 & \cdots & 0 & -\Im(\alpha^2 + \beta^2) &\Re(\alpha^2-\beta^2)
\end{pmatrix}.
$$
The centralizer of its image is $Z = \Big\{\diag\big(\det P, P, \det P, \det P\big),\ P \in O(n-1) \Big\}$. Again, this implies that $\Hom(\Gamma_g, Z)$ has $2^{2g+1}$ connected components.

$\bullet \quad \SO^*(2n)$: Here $\p^{1,0} \mapsto W' = \Big\{\begin{pmatrix}A&iA\\iA&-A\end{pmatrix}\Big\}$, where this time $A$ is complex \emph{skew-symmetric}, such that $A^*A = -\bar{A}A = \lambda I_n$, if $n$ is even, and has $n-1$ equal eigenvalues otherwise. A preferred choice for even $n$ is $A = J = \begin{pmatrix} 0 & I_{n/2}\\-I_{n/2}&0\end{pmatrix}$, and for odd $n$ one adds one row and one column of zeros. The corresponding Lie algebra morphism is the one shown in Table \ref{tab:representations}. In this case, the corresponding Lie group homomorphism involves high degree polynomials, so we restrain from writing it down, as it would not be very informative. The centralizer of its image is isomorphic to the symplectic group $\Sp(n)$, and explicitly its Lie algebra is the first factor in the following decomposition of $\k$:
\begin{equation*}
\hspace{-0.6cm}
\k = \Bigg\{\begin{pmatrix}
             B&C&D&E\\
             -C&B&E&-D\\
             D&-E&B&C\\
             -E&D&-C&B
            \end{pmatrix},
	\begin{array}{l} B^t=-B,\\ C^t=C,\\ D^t=D,\\ E^t=E\end{array}
    \Bigg \}
   \oplus
    \Bigg\{\begin{pmatrix}
             B&C&D&E\\
             C&-B&-E&D\\
             -D&-E&B&C\\
             E&-D&C&-B
            \end{pmatrix},
         \begin{array}{l} B^t=-B,\\ C^t=-C,\\ D^t=D,\\ E^t=-E\end{array}
    \Bigg \}.
\end{equation*}
As in the case of $\Sp(2n,\R)$, using the formula $\Ad(\exp) = e^{\ad}$ for element in the second factors only, and applying it to $J$ as above, a long computation proves that the action is transitive. Furthermore, the space $\Hom(\Gamma_g, \Sp(n))$ is connected, since $\Hom(\Gamma_g, \Sp(n,\C))$ is (see \cite{Ra75}, Proposition 4.2). The analysis for odd $n$ is more cumbersome but it follows the same ideas, so it will be omitted.

\section{Other Milnor--Wood inequalities}\label{sec:other}

The study of the fundamental group of a Kähler (or projective) manifold is generally carried through under some hypothesis of negative curvature, or, from the algebraic point of view, positivity of the canonical bundle. Indeed, at least in the algebraic case, one can always arrange things in order to work with a general type variety (see, for example, the introduction of \cite{KoMa10}). On the other hand, non-negativity of the curvature is very special, as our proof of the Milnor--Wood inequality underlines: By Proposition \ref{prop:MWinequality}, if $\Ric(X) \geq 0$ and $\rho$ is $\pm$-holomorphic, then $\tau(\rho) = E(\rho) = 0$, suggesting that the Toledo invariant should be trivial for these Kähler groups. This is indeed the case: If one assumes further that $X$ has non-negative holomorphic sectional curvature, then one can deform any $\rho$ to a $\C$-VHS (see \cite{Si92}), and apply Theorem 2 in \cite{Ro80} to the period mapping (that is holomorphic, with values in the period domain that has negative holomorphic sectional curvature, but is non-Kähler) to obtain $\tau \equiv 0$. However, under this additional hypothesis, much stronger results are known on $X$ and its fundamental group, see \cite{DePeSc94}. Using a recent result by Biswas and Florentino \cite{BiFl14} one can prove that the Toledo invariant is trivial for a class of Kähler manifolds containing that of non-negative Ricci curvature:

\begin{theorem}[\cite{BiFl14}]
 Suppose that $X$ is a compact Kähler manifold and that $\Gamma = \pi_1(X)$ is virtually nilpotent. Let $G^c$ be a complex Lie group, and $(\calE, \Phi)$ be a $G^c$-Higgs bundle. Then
 $$
 \lim_{t \to 0} (\calE, t\Phi) = (\calE, 0).
 $$
 This gives a homotopy retraction of $\M(X, G^c)$ to $\M(X, H)$, where $H$ is the chosen maximal compact subgroup of $G^c$.
\end{theorem}
With this result at hand, we can easily prove the following:
\begin{prop}
 Let $X$ be a compact Kähler manifold such that $\Ric(X) \geq 0$, or, more generally, such that $\Gamma = \pi_1(X)$ is nilpotent. Then, for every Hermitian Lie group $G$ and any $\rho \colon \Gamma \to G$, $\tau(\rho) = 0$.
\end{prop}
\begin{proof}
 The fact that compact manifolds with non-negative Ricci curvature have virtually nilpotent fundamental groups is due to Milnor, see \cite{Mi58}. Given $G$, the moduli space of representations $\M(X, G)$ is a closed submanifold of the moduli space $\M(X, G^c)$ of representation in the complexification $G^c$ of $G$. Furthermore, this subspace is preserved by the $\C^*$-action (see, for example, \cite{Xi00}, Proposition 3.1, for a proof of the case $G = U(p,q)$ and $\dim(X) = 1$, but the result is true in general). As a consequence, every representation in this space is deformable (through representations taking values in $G$) to a unitary one, that is, one such that $\tau(\rho) = E(\rho) = 0$.
\end{proof}
\begin{cor}
 No cocompact lattice in a Hermitian Lie group $\Gamma < G'$ is solvable.
\end{cor}
\begin{proof}
 By a Theorem of Delzant \cite{De10}, a (virtually) solvable Kähler group is virtually nilpotent, so the preceding corollary applies. The inclusion $\Gamma \to G'$, then, gives a representation with Toledo invariant $\tau = \Vol(\Gamma \backslash G'/K')$, a contradiction.
\end{proof}

We conclude this section with some discussion on the restrictiveness of the hypothesis ``being deformable to a $\pm$-holomorphic map''. In a series of papers (see \cite{BrGPGo06} for an overview), Bradlow, García-Prada and Gothen proved that for most simple Hermitian Lie groups $G$ every representation of a surface group can be deformed to a $\pm$-holomorphic one (this is true for $G = \U(p,q)$, $G = \SO^*(2n)$ or $G = \SO(n,2)$, if $n \geq 4$). This motivated us to introduce the condition of being deformable to a $\pm$-holomorphic representation, as possibly not too restrictive. They prove this result (see, for example, \cite{BrGPGo03} for $\U(p,q)$) by considering a minimum of the energy functional $E$, which exists because of properness, and proving, thanks to a formula by Hitchin (see \cite{Hi92}, \S 9) and Riemann-Roch, that simple minima must be $\pm$-holomorphic. Properness and Hitchin's formula hold in higher dimension as well, the latter being proved in \cite{Sp14}, Theorem 7.6; however, although the quest for similar topological results was part of our motivation in proving that theorem, the proof by Bradlow--García-Prada--Gothen does not carry through in higher dimension, the key point being that for manifolds of general type $\chi(\calO_X) < 0$ if $\dim(X) = 1$ and $\chi(\calO_X) > 0$ if $\dim(X) > 1$. Remark that proving this result for stable Higgs bundles would suffice, since the general polystable case would follow (the Toledo invariant is additive on direct sums, and maximality implies that all signs must agree). Remark, however, that in higher dimension it is \emph{not} true that every representation may be deformed to a $\pm$-holomorphic one: The example studied by Kim, Klingler and Pansu \cite{KiKlPa12} gives a locally rigid representation $U(n,1) \to U(2n,2)$ such that $\tau(\rho) = 0$ but that is not unitary, hence not $\pm$-holomorphic. However, this is the direct sum of something holomorphic and something anti-holomorphic, and this phenomenon cannot happen for maximal representation.

Finally, let us link the (deformability to) $\pm$-holomorphic representation to dimensional reduction. Suppose that $X$ is projective, and even that $X$ is Kähler-Einstein, so that the Kähler class is as in \eqref{eqn:KE}. Then, taking $n-1$ hyperplane sections in general position, we obtain a smooth curve $i \colon \Sigma \subset X$. The inclusion is submersive on fundamental groups, $i_* \colon \pi_1(\Sigma) \twoheadrightarrow \Gamma = \pi_1(X)$, and gives an embedding of representation varieties. Up to some constants, our definition of the Toledo invariant of $\rho \colon \Gamma \to G$ is just the Toledo invariant of the induced representation $i^*\rho \colon \pi_1(\Sigma) \to G$. The same is true for $E(\rho)$, so one might hope to restrict a minimum of $E$ to obtain a minimum on the bigger space $\Hom(\pi_1(\Sigma), G)$ (this is actually equivalent to our thesis: If a minimum in $\Hom(\Gamma, G)$ is holomorphic, then the restriction must be holomorphic, hence a minimum in $\Hom(\pi_1(\Sigma),G)$). The exact relations are:
$$
\tau(\rho) = \frac{1}{n!} \bigg(\frac{4\pi}{m \cdot c_X}\bigg)^{n-1} \tau(i^*\rho), \quad E(\rho) = \frac{1}{(n-1)!} \bigg(\frac{4\pi}{m\cdot c_X}\bigg)^{n-1} E(i^*\rho).
$$
Remark that in particular for an $n$-dimensional locally symmetric $X$ (for which $c_X$ is an integer), $\tau(\rho) \in \Q \pi^n$ (and since $\tau(i^*\rho) \in 4\Z \pi$, this can be used to study the divisibility of $\tau(\rho)$). Note also that a maximal representation of $\Gamma$ will never restrict to a maximal representation of $\pi_1(\Sigma)$. The short motivation for this is that, by \cite{BuIoWi10}, maximal representations are faithful and discrete, hence this would imply that $\pi_1(\Sigma) \cong \Gamma$. But we can be more explicit: Suppose that $X$ is a compact quotient of the complex 2-ball $\B^2$, let $\rho \colon \Gamma \to G$ be a representation, and suppose that $\Sigma \subset X$ is cut out by the very ample divisor $K_X$. Then by Riemann--Roch for surfaces, $\chi(\Sigma) = -K_X^2 = -\frac{9}{8 \pi^2} \Vol(X)$ (recall that $c_X = 3$ in this case). The Milnor-Wood inequality of Proposition \ref{prop:MWinequality} (that in this case agrees with the one by Burger-Iozzi) gives $|\tau(\rho)| \leq \rk(G) \Vol(X)$. However, applying the classical Milnor-Wood inequality to $i^*\rho$, one obtains
$$
\big|\tau(\rho)\big| = \frac{2\pi}{3}|\tau(i^*\rho)| \leq \frac{8\pi^2}{3}\rk(G) |\chi(\Sigma)| = 3 \rk(G) \Vol(X).
$$
In particular, maximal representations of $\Gamma < \SU(2,1)$ restrict to representations of $\pi_1(\Sigma)$ having Toledo invariant equal to one third of the maximal Toledo invariant $\tau_{\max} = 2\pi \rk(G) |\chi(\Sigma)|$.

\bibliographystyle{alpha}
\bibliography{refs}

\end{document}